\documentclass[12pt, a4paper]{amsart}
\usepackage{amsmath,amssymb,amscd,amsfonts}
\newtheorem{theorem}{Theorem}[section]
\newtheorem{rem}[theorem]{Remark}
\newtheorem{defn}[theorem]{Definition}
\newtheorem{ex}[theorem]{Example}
\newtheorem{lemma}[theorem]{Lemma}

\newtheorem{proposition}[theorem]{Proposition}

\newtheorem{corollary}[theorem]{Corollary}

\newcommand\Q{{\mathbb{Q}}}

\def\det{\mathop{\rm det}\nolimits}

\def\Supp{\mathop{\rm Supp}\nolimits}
\def\tr{\mathop{\rm tr}\nolimits}

\def\Aut{\mathop{\rm Aut}\nolimits}

\def\cO{{\mathcal O}}
\def\cI{{\mathcal I}}

\def\cR{{\mathcal R}}
\def\cG{{\mathcal G}}

\def\cF{{\mathcal F}}
\def\cL{{\mathcal L}}

\def\cM{{\mathcal M}}
\def\cV{{\mathcal V}}

\let\ol=\overline
\let\wt=\widetilde

\begin{document}
\title[]{Orbifold generic semi-positivity: \\
an application to
families of canonically polarized manifolds}

\author{Fr\'ed\'eric Campana, Mihai P\u aun}

\address{Institut Elie Cartan \\
Universit\'e Henri Poincar\'e\\
B. P. 70239, F-54506 Vandoeuvre-l\`es-Nancy Cedex, France\\
et: Institut Universitaire de France}
\email{frederic.campana@univ-lorraine.fr}

\address{Korea Institute for Advanced Study \\
School of Mathematics\\
85 Hoegiro, Dongdaemun-gu, Seoul 130-722, Korea.}
\email{paun@kias.re.kr}

\begin{abstract} 
Let $X$ be a normal projective manifold, equipped with an effective `orbifold'
divisor $\Delta$, such that the pair $(X, \Delta)$ is log-canonical. 
We first define the notion of `orbifold cotangent bundle' $\Omega^1(X, \Delta)$, 
living on any suitable ramified 
cover of $X$. We are then in position to formulate and prove (in a completely different way) an orbifold version of Y. Miyaoka's generic semi-positivity theorem: 
$\Omega^1(X, \Delta)$ is generically semi-positive if $K_X+ \Delta$ is pseudo-effective.
Using the deep results of the LMMP, we immediately get a statement conjectured by E. Viehweg: if $X$ is smooth, and if $\Delta$ is a reduced divisor with simple normal crossings on $X$ such that some tensor power of $\Omega^1(X, \Delta)=\Omega^1_X(Log (\Delta))$ contains the injective image of a big line bundle, then $K_X+ \Delta$ is big. 

This implies, by fundamental results of Viehweg-Zuo, the `Sha\-farevich-Viehweg hyperbolicity conjecture': if an algebraic family of canonically polarized manifolds parametrised by a quasi-projective manifold $B$ has `maximal variation', then $B$ is of log-general type.

\

\end{abstract}

\maketitle

{\bf R\'esum\'e:} Nous d\'efinissons la notion de `fibr\'e cotangent orbifolde' $\Omega^1(X, \Delta)$ pour une paire $(X,\Delta)$ log-canonique: ce fibr\'e est d\'efini sur des rev\^etement cycliques ad\'equats. Nous formulons et d\'emontrons ensuite une version orbifolde du th\'eor\`eme de semi-positivit\'e g\'en\'erique de Y. Miyaoka: $\Omega^1(X, \Delta)$ est g\'en\'eriquement semi-positif si $K_X+\Delta$ est pseudo-effectif. Nous en d\'eduisons, \`a l'aide des r\'esultats r\'ecents du PMML, un \'enonc\'e conjectur\'e par E. Viehweg: si $X$ est lisse, et si $\Delta$ est un diviseur r\'eduit \`a croisements normaux simples sur $X$ tel qu'une puissance tensorielle de $\Omega^1_X(Log(\Delta))$ contienne un fibr\'e en droites `big', alors $K_X+\Delta$ est lui-m\^eme `big'. Les travaux de Viehweg-Zuo impliquent alors la conjecture d'hyperbolicit\'e de V.I. Shafarevich: si une famille alg\'ebrique de vari\'et\'es projectives canoniquement polaris\'ees et param\'etr\'ee par une vari\'et\'e quasi-projective irr\'eductible lisse $B$ a une `variation' maximale, \'egale \`a $dim(B)$, alors $B$ est de type log-g\'en\'eral.

{\bf Mots-cl\'e:} Fibr\'e cotangent orbifolde, semi-positivit\'e g\'en\'erique, vari\'et\'e canoniquement polaris\'ees.


{\bf Classification:} 14D22,14E30,14J40,32J25.

{\bf Titre courant:} Orbifold generic semi-positivity

\section{The cotangent sheaf of an orbifold pair}

Let $X$ be a complex projective normal and connected variety of dimension $n= \dim(X)$, with $U\subset X$ a Zariski open non-empty subset contained in the smooth locus of $X$ and such that $X-U$ is of complex codimension at least $2$ in $X$ (we will in general have to shrink $U$ a finite number of times in the course of the proof, and the letter $U$ is reserved for such an appropriately chosen open subset where everything will take place). We denote by $T_U$ 
the tangent bundle of $U$ and by $\Omega^1_U$ its dual, the cotangent bundle. 
The canonical bundle of $U$ is denoted as usual by $K_U:=\det (\Omega^1_U)$.

We consider (using the terminology of \cite{Ca04}) an \emph{orbifold} divisor $\Delta:= \sum_{j=1}^r \delta_j D_j$, where
the coefficients $(\delta_j)_{j=1,..., r}$ are positive rational numbers in the interval $]0,1]$, and the $D_j's$ are irreducible, pairwise distinct hypersurfaces of $X$. We say that an \emph{orbifold pair} $(X, \Delta)$ is smooth if $X$ is smooth, and the support 
$\Supp(\Delta)=\lceil \Delta \rceil=\cup_{j=1}^r \lceil\delta_j\rceil D_j$ has normal crossings.

Such orbifold pairs are usually simply called \emph{pairs} in the LMMP, which considers only the canonical bundle $K_X+\Delta$. The motivation for introducing the 
orbifold pairs in \cite{Ca04} is to encode the multiple fibers of algebraic
fibrations in an orbifold divisor on the base, which amounts to perform a virtual ramified cover of the actual base, with ramification orders equal to the multiplicities of the fibres over the generic point of the components of the orbifold divisor. Base-changing the given fibration y this virtual cover then eliminates the multiple fibres in codimension $1$.

This construction permits to introduce a geometry on the orbifold pairs, related to, but different from, the classical theory of orbifolds. Indeed, in the classical situation we have $\delta_j= 1-\frac{1}{m_j}$,
where the coefficients $m_j\geq 1$ are integers, or $+\infty$, hence $\Delta$ then appears as the ramification divisor of some virtual ramified cover of $X$ branching along $D_j$ with multiplicity $m_j$. For general rational multiplicities, this construction needs a small adaptation.

The orbifold pairs $(X, \Delta)$ interpolate between the \emph{compact case} where, for all $j$, $m_j=1$ 
and the \emph{logarithmic case}, where these are all: $m_j= \infty$, respectively. In both (smooth) cases, we have the notions of tangent bundle, cotangent bundle and more generally, of holomorphic tensors.
\smallskip

Our first aim here is to introduce these notions for an arbitrary orbifold pair $(X, \Delta)$. In contrast to the above two cases however, the corresponding object does not live on $X$ but only on some ramified cover of $X$ as a coherent sheaf of $\cO_X$-modules at least\footnote{Although this object might possibly be defined intrinsically as on $X$ itself by introducing more general structure sheaves, proving our main result requires the consideration of such covers in order to use only the usual structure sheaf $\cO_X$.}. We shall introduce these objects first locally in coordinates, and then globalize them on some (non-canonically defined) ramified cover.

\subsection{Local construction}

\

We first assume that we are working in local coordinates $(x)=(x_1,x_2,...,x_n)$, near a smooth point of $X$ where the support of $\Delta$ is of normal crossings, and contained in the union of the coordinate hyperplanes $D_k$ defined by $x_k=0, k=1,\dots,n$. Such points cover a Zariski open subset of $X$ with complement of codimension at least two  (empty if $(X,\Delta)$ is smooth). We denote with $m_k=\frac{a_k}{b_k}$ the multiplicities in $\Delta$ of the coordinates hyperplanes. Here $a_k,b_k$ are coprime integers with $a_k=b_k=1$ if $m_k=1$ (i.e. if the coefficient $\delta_k=0$), while $a_k=1,b_k=0$ if $m_k=+\infty$ (i.e. if $\delta_k=1$).

In this case the very simple idea idea is that $\Omega^1(X,\Delta)$ should be the locally free $\cO_X$-module generated by the elements $\frac{dx_k}{x_k^{\delta_k}}=x^{1-\delta_k}.\frac{dx_k}{x_k}$, for $k=1,\dots,n$. When $\delta_k=0$, or $1$, we recover the usual `compact' and `purely logarithmic' cases.

However, this construction does not make sense in the frame of classical complex geometry. We thus need to make ramified covers in order to work in this frame.

For each coordinate hyperplane $x_k=0$, write its multiplicity as: $\frac{1}{1-\delta_k}:=m_k=\frac{a_k}{b_k}$, where $a_k,b_k$ are nonnegative coprime integers. If $\delta_k=0$, ie. $m_k=1$, we thus have: $a_k=b_k=1$, while if $\delta_k=1$, ie. $m_k=+\infty$, we have: $a_k=1,b_k=0$. In the other cases we have $a_k>b_k>0$. 

Consider now the following (local near $(0,\dots,0)$) ramified cover: $\pi: Y:=\Bbb C^n\to \Bbb C^n$ given by: $\pi(y_1,\dots,y_n):=(x_1:=y_1^{a_1},\dots,y_n^{a_n})$.
This cover ramifies at order $a_k$ over each of the coordinate hyperplanes $x_k=0$. It thus does not ramify at all over the divisors where $\Delta$ is either $0$, or `purely logarithmic'.

Pulling back our `orbifold' one-forms $\frac{dx_k}{x_k^{\delta_k}}$ by $f$, we get (up to a non-zero constant factor) the holomorphic or logarithmic one-forms $\pi^*(\frac{dx_k}{x_k^{\delta_k}})=y_k^{b_k}\frac{dy_k}{y_k}$. 

Slightly more generally, if we consider a ramified cover defined by $\pi(y_1,\dots,y_n)=(y_1^{g_1.a_1},\dots,y_n^{g_n.a_n})$, with $g_k$ positive integers, we would obtain: $\pi^*(\frac{dx_k}{x_k^{\delta_k}})=y_k^{g_k.b_k}\frac{dy_k}{y_k}$.

\

The following alternative coordinate-free description was suggested to us by Stefan Kebekus: $\pi^*(\Omega^1(X,\Delta))= [\pi^*(\Omega^1_X(\delta_1.D))\cap \Omega^1_Y(Log D')]$ if $D':=\pi^{-1}(D)$, valid near smooth points of $X$ where the support $D$ of $\Delta$ is smooth and defined by $x_1=0$.

\

The dual sheaf $T(X,\Delta)$ will be defined similarly. In the same coordinates, it is `virtually' generated by the elements $x_k^{\delta_k}. \frac{\partial}{\partial x_k}$. On $Y$, they become the dual meromorphic vector fields $y_k^{(1-g_k.b_k)}.\frac{\partial}{\partial y_k}$.

Observe that the sheaves defined in this way do not depend on the choice of coordinates, provided these are `adapted' to $\Delta$.

In this situation, we define the inverse image of the `cotangent sheaf $\Omega^1(X,\Delta)$' by $f$ to be the (locally free) sheaf of $\cO_Y$-modules generated by the elements $\pi^*(\frac{dx_k}{x_k^{\delta_k}})$ just computed. For the rest of this article, we shall denote it by 
$\pi^*(\Omega^1(X,\Delta)).$ We proceed similarly in order to define its dual $\pi^*(T(X,\Delta))$, and more generally, any
tensor sheaf associated to $(X, \Delta)$. 

Notice that no such inverse image sheaf is presently defined at the points of $X$ which are either not smooth, or where the support of $\Delta$ is not of normal crossings. This is indeed not needed, for our purposes (which permit to ignore codimension two subsets). However a (much more involved) definition could be given at these points too, but involving further considerations.

 We shall next globalize this inverse image by considering global ramified covers of $X$. Normal cyclic covers will be sufficient here. We shall briefly explain how smooth Kummer covers can be used to get locally free inverse image sheaves which are everywhere defined by the above formulae, when $(X,\Delta)$ is smooth. Such covers have also be introduced by A. Langer for similar purposes in the surface case (\cite{La}), and also in \cite{JK},\S.2, in the case of integral multiplicities.

\subsection{Global construction}

\

\

 Let $\Delta:=\sum_j\delta_j.D_j$ be an orbifold divisor, with $\delta_j=1-\frac{1}{m_j}, m_j=\frac{a_j}{b_j}$ as above. Let $D_1,\dots,D_m$ be the support of the `finite' part of $\Delta$ (i.e. those $D_j$ such that $0<\delta_j<1$, or equivalently, such that $1<m_j<+\infty$). Let $a$ be the least common multiple of the $a_j, j=1,\dots,m$.
 
  There exists a very ample line bundle $H$ on $X$, and a positive integer $g'$ such that $g'.a.H-(D_1+\dots+D_m)$ has a non-zero section with a reduced zero locus $E$ in codimension one (this can be seen, for example, by applying the same statement to a smooth model $s:X_1\to X$ of $X$, and to the strict transform of $(D_1+\dots+D_m)$ in $X_1$, using the fact that $s^*(H)=H_1+E'$, for $H_1$ ample on $X_1$, and $E'$ an effective $s$-exceptional divisor). 
  
  We consider the normalization $\pi:Y\to X$ of the cyclic cover of $X$ associated to the section $E+ (D_1+\dots+D_m)$ of $g.H, g:=g'.a$, and define $\pi_U^*(\Omega^1(X,\Delta))$ as in the preceding section over the Zariski open subset $U$ of $X$ consisting of the points where $X$ is smooth, and $E+\lceil\Delta\rceil$ is a divisor of normal crossings. This definition makes sense, since $\pi$ ramifies over the generic point of each $D_j, j=1,\dots,m$, to the order $g:=g'.a$, which is divisible by $a_j$. Since this sheaf is defined algebraically over $\pi^{-1}(U)$, it has a coherent extension $(i_U)_*(\pi_U^*(\Omega^1(X,\Delta)))$ to $Y$, denoted\footnote{This ad hoc definition will be sufficient for our present purposes.}  $\pi^*(\Omega^1(X,\Delta))$.
  
  Let $G\cong \Bbb Z_g$ be the Galois group of the covering $\pi$. The sheaf $\pi^*(\Omega^1(X,\Delta))$ is, by construction, invariant under the natural action of $G$ over $U$, which extends to $\pi^*(\Omega^1(X,\Delta))$, by its very definition.

  The dual sheaf $\pi^*(T(X,\Delta))$ is defined similarly, as in the local description above over $U$, and extended to $Y$ by applying $(i_U)_*$ also.
  
  \begin{rem}\label{k} If $(X,\Delta)$ is smooth, and if the support of $\Delta$ is of \emph{simple} normal crossings, we can obtain from \cite{KMM87} (see equally \cite{EV} and th references therein) a (non-cyclic) finite cover $\pi:Y\to X$ with $Y$ smooth, and a branching divisor $B+(D_1+\dots+D_m)$ on $X$ which is of simple normal crossings by using a composition of such cyclic covers, one for each the the $D_j, j=1,\dots,m$. In this case, $U=X$, so no extension $(i_U)_*$ is needed, and $\pi^*(\Omega^1(X,\Delta))$ is a locally free sheaf on $Y$, inductively generated by the explicit elements given in coordinates in the preceding section. \end{rem}

\begin{defn} Let $Y$ be a normal and connected complex projective variety, and 
let $G\subset \Aut(Y)$ be a finite group of automorphisms of $Y$. Let $U$ be a $G$-invariant Zariski open subset contained in the smooth locus of $Y$, and $\cF_U\subset\cM(T_Y)$ be a coherent subsheaf of the sheaf of
meromorphic vector fields on $U$, such that $\cF=(i_U)_*(\cF_U)$ is a coherent sheaf of $\cO_Y$-modules.

Then we say that $\cF$ is $G$-invariant if for each open set $V\subset Y$, the 
differential of each element $h\in G$ induces over $U$ an isomorphism between 
the space of sections of $\displaystyle \cF|_V$ and the space of sections of $\displaystyle \cF|_{h(V)}$. This action then extends to $\cF$ over all of $Y$.
\end{defn}

 We will need the following fact, which is (likely) well-known in different contexts\footnote{In particular, it holds true for $\pi^*(\cG)$, $\cG$ any coherent sheaf $\cG$ on $X$, not only for $\cG=T_X$.}.

\begin{lemma} Let $\pi:Y\to X$ be the preceding cyclic cover defined above, with Galois group $G\cong \Bbb Z_g$. Let $\cF\subset \pi^*( T_X)$ be a $G$-invariant coherent $\displaystyle \cO_{Y}$-module, which is saturated inside the 
inverse image $\pi^*(T_X)$ of the tangent sheaf $T_X$. Then $\cF= \pi^* (\cF_X)$ for some coherent sheaf $\cF_X$ of $\cO_X$-modules on $X$.
\end{lemma}

\begin{proof} It will be sufficient to construct $\cF_X$ over a Zariski open subset with complement of codimension at least two, and to consider its extension to $X$. We shall thus consider a smooth point $x_0$ of $X$ where the support of $\Delta$ is smooth, and thus consists of a single $D_1$ of local equation $x_1=0$. If $y_0$ is a point of $Y$ lying over $x_0$, in suitable coordinates, $\pi:Y\to X$ is given near $y_0$ by: $$(t, y_2,\dots, y_n)\to (t^g, y_2,\dots, y_n)$$
and the action of the generator $h\in G$ is given by the multiplication of the coordinate $t$ by a primitive $g$-th root $\mu$ of unity. 

We shall show that, locally, $\cF$ is generated as a $\cO_Y$-module, by $G$-invariant sections, which are thus lifts of sections of $T_X$. And $\cF_X$ will be locally generated as a $\cO_X$-module, by these sections.

Let $V$ be a local section of $\cF$ defined in a neighborhood of $x_0$. Then $V= \sum_{k=0}^{g-1}t^k\pi^*(v_k)$
for local sections $v_k$ of the sheaf $T_X$, since
$\cF\subset \pi^*(T_X)$. Since $\sum_{k=0}^{g}\mu^{jk}=0$ if $j$ is not divisible by $g$, and $h^*(t)=\mu.t$, we have:

$$\pi^*(v_0)= \frac{1}{g}.\sum_{p=0}^{g}(h^*)^p(V)$$
So, $\cF$ being $G$-invariant, we get: $\pi^*(v_0)\in \cF$.Thus: $(V-\pi^*(v_0))=t.V_1,$ with $V_1:=\sum_{k=0}^{g-2}t^{k}\pi^*(v_{k+1})$.
By our saturation assumption, $V_1$ is a section of $\cF$, since $t.V_1$ is a section of $\cF$, and $V_1$ is a section of $\pi^*(T_X)$. for $k=1,\dots,(g-1)$. By induction on $k$, we get that $\pi^*(v_k)$ is a section of $\cF$, for $k=0,1,\dots,g$.
Thus $\cF$ is generated as an $\cO_Y$-module by elements of the form $\pi^*(v)$, for $v$ local sections of $T_X$ \end{proof}

\begin{rem} If $\pi:Y\to X$ is a composition of cyclic covers, the above argument can be also applied inductively. In particular, the conclusion holds in the situation of remark \ref{k}.
\end{rem}

\begin{rem} \label{degree}The tangent and cotangent sheaf associated to $(X, \Delta)$ are clearly 
invariant by the group $G$ acting on $X$. Also, one has the inclusion of sheaves $\pi^*( T(X, \Delta))\subset f^*(T_X)$ over the Zariski open subset $U\subset X$ consisting of smooth points of $X$ where $Supp(\Delta)$ is smooth. A similar fact holds for the cotangent sheaves (with a reversed inclusion). Moreover, we have, for any projective irreducible curve $C'\subset f^{-1}(U)$ which meets transversally each component of $\pi^{-1}(Supp(\Delta))$ the exact sequence:
$$0\to f^*\Omega^1(X)_{\vert C'}\to f^*\Omega^1(X, \Delta)_{\vert C'}\to f^* \cO(\Delta)_{\vert C'}\to 0$$
on $C$; this shows in particular that the degree of $\pi^*\Omega^1(X, \Delta)$ on any curve cohomologous with the 
class $\pi^* (H)^{n-1}$, for $H$ ample on $X$, is equal to $g.(K_X+ \Delta)\cdot H^{n-1}$, since the complement of $U$ is of codimension of least $2$ in $X$.
\end{rem}

\begin{rem}  It is immediate to see that the inverse image by $\pi$ of any section of $S^{[r]}\Omega^1{(X,\Delta)}$, as defined in \cite{Ca07}, over an open subset $V\subset U$ defines a $G$-invariant section of 
$\otimes^r\pi^*\Omega^1(X,\Delta)$  over $\pi^{-1}(V)$.

\end{rem}

\subsection{Notion of orbifold generic semi-positivity}

\begin{defn}\label{defgsp} We consider the data $(X,\Delta), H, f, Y$ as above, with $\pi:Y\to X$ a cyclic cover adapted to our situation, constructed as in the beginning of \S.1.2. We shall say that $\Omega^1(X,\Delta)$ is $\pi$-generically semi positive (gsp in abbreviated form) if for any polarization $B$ on $X$, the sheaf $\pi^*\Omega^1(X,\Delta)$ defined above 
is generically semi-positive with respect to $\pi^*(B)$ in the usual sense. The latter condition means 
that any quotient subsheaf $\cG$ of $\pi^*\Omega^1(X,\Delta)$ has nonnegative degree on $(\pi^*(B))^{n-1}$.

\end{defn}

\begin{rem} This notion depends only on Zariski open subsets $U$ with complements of codimension at least $2$ in $X$ (which is the reason why we did not need to have a refined definition of $\pi^*(\Omega^1(X,\Delta))$ over the complement of such a $U$). 

We shall also see later (see remark \ref{rwu}.(2) below) that this notion of \emph{generic semi-positivity} for orbifold cotangent bundles does not depend on the choice of covers chosen for its definition. For the time being, we shall check this in the following special case, used crucially in the proof of theorem \ref{tgsp}.
\end{rem}

We shall consider the following data.

1. Let $(X,\Delta)$ be an orbifold pair, with $X$ normal and projective, and $\pi:Y\to X$ will be a cyclic cover of degree $g$ associated to $\Delta$ as above.

2. Let $f:X\dasharrow Z$, $Z$ normal, be a rational fibration. We denote by $U_f\subset X$ the Zariski open set with complement of codimension at least $2$ in $X$ consisting of smooth points $x$ of $X$ at which the support of $\Delta$ is smooth (or empty), and such that the map $f$ is holomorphic at $x$, with fibre having a smooth reduction. By blowing-up suitably $X$ and $Z$, we may and shall assume $f$ to be holomorphic and `neat' in the sense of \cite{Ca04}, see definition \ref{dneat} below. In this process, $U_f$ thus remains unchanged, if we restricted it so as to avoid the indeterminacy locus of this `neat' model of $f$. The image of $U_f$ (still restricting it with complement of codimension at least $2)$ may and shall be assumed to be contained in the smooth locus $Z^{reg}$ of $Z$, since the $f$-exceptional divisors of $X$ are also (by `neatness') contained in the exceptional divisor of the modification of our `initial' $X$. The sheaf $\Omega^1_X/f^*(\Omega^1_Z)$ is thus well-defined in the usual sense over $U_f$.

3. Let $C'\subset Y$ will be a generic member of the algebraic family of complete intersections $\pi^*(m.B)^{n-1}, m$ sufficiently large: $C'$ thus a projective smooth connected curve contained in $\pi^{-1}(U_f)$ meeting transversally each component of $\pi^{-1}(Supp(\Delta))$. $C'$ also meets transversally each of the finitely many irreducible divisors $F_k$ of $X, k=1,\dots,r$ such that $f(F_k)$ is a divisor of $Z$, with multiplicity of $f$ along $F_k\cap U_f$ equal to some $t_k\geq 2$. Let $C$ be the normalisation of its image in $X$.

4. Let $\Delta^{\rm hor}$ be the union of the components of $\Delta$ which meet $U_f$ and are mapped by $f$ onto $Z$, each affected with the same coefficient it has in $\Delta$.

\begin{proposition}\label{deg} In this situation, let $\cF_X:=f^*(\Omega^1_Z)\subset \Omega^1_X$: this is a well-defined coherent sheaf on $U_f$. Let $\cF^{\Delta}\subset \pi^*(\Omega^1(X,\Delta))$ be the saturation of $\pi^*(\cF_X)$ in $\pi^*(\Omega^1(X,\Delta))$. Let $Q_{f,\Delta}$ be the quotient sheaf $\pi^*(\Omega^1(X,\Delta))/\cF^{\Delta}$.

Then: $\frac{1}{g.m^{n-1}}.deg_{C'}(Q_{f,\Delta})=(K_{X/Z}+\Delta).C-[\sum_{k=1}^{k=r} (t_{F_k}-\frac{1}{m_{\Delta}(F_k)}).F_k].C$, where $m_{\Delta}(F_k)\geq 1$ is the multiplicity of $F_k$ in $\Delta$. This equality can also be written as: $(\frac{1}{g.m^{n-1}}).deg_{C'}(Q_{f,\Delta})=[K_{X/Z}+\Delta-D(f,\Delta)].H^{n-1}$, if $D(f,\Delta):=[\sum_{k=1}^{k=r} (t_{F_k}-\frac{1}{m_{\Delta}(F_k)}).F_k]$.

\end{proposition}

\begin{proof}  Before starting the proof, let us notice an ambiguity in the notations: the symbol $f^*(\Omega^1_Z)$ denotes the composition $(df)\circ f^*(\Omega^1_Z)$, where $f^*$ is just the inverse image sheaf on $X$, while $df$ is the differential mapping $f^*(\Omega^1_Z)$ into $\Omega^1_X$. By contrast, $\pi^*(\cF_X)$ is just the inverse image sheaf on $Y$, not composed with the differential $d\pi$. Thus, in particular, the ramification of $\pi$ along the divisor $\pi^{-1}(E)$ is not taken into account in the computation below, where $E$ is the codimension one set defined in \S1.2.

The quotient $\cF^{\Delta}/\pi^*(\cF_X)$ is a skyscraper sheaf concentrated on the union of the support of $\Delta$, and of the $F_k$ (this over $U_f$, at least), and: $\pi^*(K_{X/Z}+\Delta).C'-det_{C'}(Q_{f,\Delta}).C'$ is equal the length of this skyscraper sheaf over $C'$. We are thus reduced to the local computation of this length at an arbitrary point $y_0\in C'$. By the transversality assumption, we may assume that we have local coordinates $y:=(y_1,\dots,y_n)$ and $x:=(x_1,\dots,x_n)$ near $y_0$ and $x_0:=\pi(y_0)$ respectively such that, in these coordinates: $\pi(y)=(y_1^g,y_2,\dots,y_n)$, and $f(x_1,\dots,x_p,x_{p+1},\dots, x_n)=(z_1:=x_1^t,z_2:=x_2,\dots,z_p:=x_p)$, if $p:=dim(Z)$, while the curve $C'$ is parametrically defined by the map $\gamma:w\to \gamma(w):=(w,0,\dots,0)\in Y$, for $w\in \Bbb C$ near $0$.  The coordinates $x,y,z$ with indices $2$ or more do not contribute to the computation, and are thus ignored; the sheaf $\gamma^*(\pi^*(f^*(\Omega^1_Z)))$ is thus generated by $w^{t.g-1}.dw$, while its saturation in $\pi^*(\Omega^1(X,\Delta))$ is generated by $\gamma^*(y_1^{g.(1-\delta)-1}.dw)=w^{g.(1-\delta)-1}.dw$, if $\delta$ is the $\Delta$-multiplicity of the divisor $D_1$ of local equation $x_1=0$ in $X$. This establishes the claim, since the local length at $y_0$ is then given by: $(g.t-1)-(g.(1-\delta)-1)=g.(t-\frac{1}{m})$, if $m=(1-\delta)^{-1}$ is the $\Delta$-multiplicity of $D_1$. \end{proof}

\begin{rem}
In particular, we see from this formula that the intersection number we compute is 
completely independent of the very ample hyperplane section we have used in order to
define $\pi:Y\to X$ and $\pi^*(\Omega^1(X,\Delta))$, although the map $\pi$ is ramified along $H$.
\end{rem}

\begin{rem}\label{ind} This proposition thus shows that the degree of `algebraically defined' quotient sheaves of $\pi^*(\Omega^1(X,\Delta))$ on `generic' curves of $Y$ is, in fact, computed from data defined on $X$, and thus independent on the cyclic cover $Y$. The first step of the proof of theorem \ref{tgsp} will, in fact, precisely show that such quotients are `algebraically defined' if anti-ample on Mehta-Ramanathan curves.

Remark also that proposition \ref{deg} holds true for \emph{any birational model} of $f$, provided one chooses $C'$ accordingly. In particular, we may (and shall in the end of the proof of theorem \ref{tgsp}) assume that $f:X\to Y$ is holomorphic and $\Delta$-neat, in the sense of definition \ref{dneat} below.
\end{rem}

\section{An orbifold version of Miyaoka's generic semipositivity}
\medskip 

A $\Bbb Q$ divisor $E$ on a projective normal variety $X$ is said to be \emph{pseudo-effective} if
the divisor $E+\varepsilon.H$ is $\Bbb Q$-effective (and thus big) for any rational $\varepsilon>0$. According to \cite{BDPP}, 
$E$ is pseudo-effective if and only if $E.C\geq 0$, for any irreducible member $C\subset X$ of any covering family of curves 
on $X$.

\begin{theorem}\label{tgsp} The sheaf $\pi^*\Omega^1(X,\Delta)$ is $\pi$-generically semi-positive if the pair $(X,\Delta)$ is log-canonical, and $K_X+\Delta$ is pseudo-effective on $X$. 
\end{theorem}

\begin{rem} The proof in fact shows, more precisely, that if the pair $(X,\Delta)$ is log-canonical, and if the sheaf $\pi^*\Omega^1(X,\Delta)$ is not $\pi$-generically semi-positive, there exists a `neat' dominant fibration $f: (X,\Delta)\to Z$ (on some suitable birational model of $(X,\Delta))$ such that $K_{X}+\Delta$ is not pseudo-effective on the generic fibre $X_z$ of $f$. The dimension of $X_z$ is the rank of the largest semi-stable quotient of minimal slope of $\pi^*\Omega^1(X,\Delta)$ relative to some polarisation $\pi^*(H)$ of $Y$ such that $(K_X+\Delta).H^{n-1}<0$. In particular, these fibrations are all constant maps if, for all such polarisations, $\pi^*\Omega^1(X,\Delta)$ is semi-stable.
\end{rem}

We shall need the following immediate generalisation, deduced from the fact that the tensor powers of nef bundles on a smooth curve are nef:

\begin{corollary}\label{cgsp} For any integer $m\geq 0$, the sheaf $\otimes^{m} \pi^*\Omega^1(X,\Delta)$ is $\pi$-generically semi-positive if the pair $(X,\Delta)$ is log-canonical, and $K_X+\Delta$ is pseudo-effective on $X$. 
\end{corollary}

The generic semi-positivity theorem of Y. Miyaoka (\cite{Mi}) asserts that if {\sl a normal projective variety  $X$ is not uniruled, then $\Omega^1(X)$ 
is generically semi-positive}\footnote{The converse is an open delicate problem.}. This statement is equivalent to the conjunction of two results:
 {\sl the bundle $\Omega^1(X)$ 
is generically semi-positive if $K_X$ is pseudo-effective} and: \emph{the canonical bundle $K_X$ is pseudo-effective if and only if $X$ is not uniruled}, respectively.

Theorem \ref{tgsp} above extends the first assertion to the orbifold situation\footnote{Under the log-canonicity assumption.}, giving when $\Delta=0$ an alternative proof in characteristic zero.

The second statement admits an orbifold counterpart, but {\it a priori} in the klt case only. This is an immediate application of \cite{BCHM}:

\begin{theorem}\label{two} $K_X+\Delta$ is pseudo effective if and only if $(X,\Delta)$ is log-canonical and not `weakly uniruled' (i.e.: covered by rational curves $R$ such that $(K_X+\Delta).R<0$).
\end{theorem}

The property of `weak-uniruledness' is, however, too weak to give interesting geometric informations. See \cite{Ca07} for more geometric (but in general only conjectural) variants of `orbifold uniruledness'.

\begin{rem}\label{rwu} 

\

1. It follows from theorem \ref{tgsp} and its proof that the property of $\pi^*(\Omega^1(X,\Delta))$ being gsp is independent of the cyclic cover used to define this property if $K_X+\Delta$ is pseudo-effective. Conversely, if $\pi^*(\Omega^1(X,\Delta))$ is not gsp for some $\pi:Y\to X$, the proof of theorem \ref{tgsp} constructs a fibration as in proposition \ref{deg} above, and this proposition shows that $(\pi')^*(\Omega^1(X,\Delta))$ will be non-gsp for every other cyclic cover $\pi'$ associated to $(X,\Delta)$.

2. The conclusion of theorem \ref{tgsp} can conjecturally be strengthened to: ``every quotient of $\Omega^1(X,\Delta)$ has a pseudo-effective determinant". Our arguments do not however permit to prove this. When $X$ is smooth and $\Delta=0$, this has been shown in \cite{CPe}.
\end{rem}

The proof of theorem \ref{tgsp} consists of the following steps: arguing by contradiction, we construct, by Harder-Narasimhan theory and Mehta-Ramanathan theorem, a foliation on $X$, as Miyaoka did,  (the involutiveness of the distribution is in our orbifold context more delicate, however). The algebraicity of the leaves is shown by applying the criterion of Bogomolov-MacQuillan (\cite{BMQ}, see also \cite{Bo},\cite{Ha}, \cite{KST}). The contradiction is obtained using a slight modification of the orbifold version of Viehweg weak-positivity of direct images of relative canonical bundles as in \cite{Ca04}, theorem 4.13.

We notice here that these two ingredients were also used in a parallel manner by Andreas H\"oring in \cite{Ho}, theorem 1.4, to show that if $X$ is a normal projective variety of dimension $n$ and $A$ a nef and big Cartier divisor on $X$ such that $K_X+ nA$ is nef, then $\Omega^1_X\otimes  A$ is generically semi-positive, unless $X$ is birationally a scroll. 

\

\noindent We now start the proof of Theorem \ref{tgsp}. 

\smallskip 

\begin{proof} We consider a cyclic cover $\pi:Y\to X$ associated to the orbifold pair $(X,\Delta)$. Arguing by contradiction, we assume the existence of a $G$-invariant torsion free sheaf of $\displaystyle \cO_{Y}$-modules, say 
$\cG_0$, which admits a surjective map 
$$\pi^*\Omega^1(X, \Delta)\to \cG_0\to 0$$
and such that $\deg_{H'}(\cG_0)<0$; here we use the notation $H':=f^*(B)$ for the (ample) inverse image of an arbitrary hyperplane section $B$ on $X$. In other words, the degree of the restriction of $\cG_0$ to any
Mehta-Ramanathan curve $C'$ relative to $H'$ is negative. 
 We can assume that $C'$ do not intersects the singular locus of $\cG_0$, that is to say, 
 that $\cG_0$ is locally free along $C'$. The dual $\cG_0^* $ of $\cG_0$, 
 is a $G$-invariant torsion free subsheaf of $\pi^*(T{(X,\Delta)})$, and $\deg_{H'}(\cG_0^*) >0$.
 By hypothesis $K_X+ \Delta$ is pseudo-effective, the degree of the determinant of $\pi^*(T{(X,\Delta)})$ on $C'$ is negative, by Remark \ref{degree}, and the orbifold tangent sheaf $\pi^*(T{(X,\Delta)})$ \emph{is not $H'$-semi-stable}. 

Let $\cF_1$ be the semi-stable piece of the Harder-Narasimhan filtration of $\pi^*(T{(X,\Delta)})$ of maximal $H'$-slope.
 By Mehta-Ramanathan, it restricts to a piece of maximal slope on the generic Mehta-Ramanathan curves $C'\subset Y$ associated to $H':=f^*(B)$.

 \begin{lemma} The sheaf $\cF_1$ is $G$-invariant and saturated in $\pi^*(T{(X,\Delta)})$. Moreover, the restriction of 
 $\cF_1$ to $C^\prime$ is semi-stable, and hence ample.
 \end{lemma}
 
 \begin{proof} The first assertion is a consequence of its maximality properties,
 together with the fact that we are considering the stability with respect to an inverse, and hence $G$-invariant, polarisation. 
 The second one
 is standard in Harder-Narasimhan theory. The third one is due to the fact that the degree is strictly positive, together with semi-stability.
 \end{proof}

 \begin{lemma} {Let $L: \wedge^2\cF_1\to f^*T{(X,\Delta)}/\cF_1$ be any $\cO_{Y}$-linear map. Then 
$L= 0$}. 
 \end{lemma}

\begin{proof}This is a consequence of the semi-stability of $\cF_1$, via an argument due to Y. Miyaoka in \cite{Mi}, resting on the fact that the slope of the wedge product is twice the slope of the factors, by semi-stability again. 
 \end{proof}

\noindent We now need to `descend' from $(X,\Delta)$ to the manifold $X$. We thus consider the saturation of $\cF_1$ in $\pi^*T_{X}$, denoted:  
 $\cF_1^{\rm (s)}\subset \pi^*T_{X}=(i_U)_*(T_U)$, for $U=X^{reg}$. We remark that both of these sheaves are $G$-invariant, and therefore so is 
$\cF_1^{\rm (s)}$. Therefore by Lemma 2.2, there exists a sheaf $\cF^{\rm (s)}\subset \cO(T_X)$
such that $\cF_1^{\rm (s)}= \pi^*( \cF^{\rm (s)})$.

\begin{lemma} \label{Lie}The sheaf $\cF^{\rm (s)}$ is closed under the Lie bracket; it thus defines a 
foliation on $X$. These statements hold on the regular part of $X$.
\end{lemma}

\medskip

To prove this lemma, we need here to carefully distinguish the Lie brackets of vector fields $\cL_X$ on $X^{reg}$ and $\cL_Y$ on $Y^{reg}$, since $\pi^*(\cO(T_X))$ is not closed under $\cL_Y$.
\medskip

Let $\cL: \Lambda^2\cF^{\rm (s)}\to T_X/\cF^{\rm (s)}$ be
deduced from the Lie bracket on the tangent bundle $T_X$. Let 
$\pi^* \cL$ be the map deduced from $\cL$ by inverse image and extension by $\cO_{Y}$-linearity; it is defined as follows:
$$\pi^*(\cL): \Lambda^2\cF_1^{\rm (s)}\to \pi^* T_X/\cF_1^{\rm (s)}.$$

Let $J:\Lambda^2 \cF_1\to \Lambda^2\cF_1^{\rm(s)}$, and $J_1:\pi^*(T(X,\Delta))/\cF_1\to \pi^*(TX)/ \cF_1^{\rm(s)}$ be the natural injections (recall that $\cF_1=\cF_1^{\rm(s)}\cap \pi^*(T(X,\Delta)))$. We just need to show that $\cL\circ J$ vanishes along any sufficiently generic complete intersection curve of large multiples of $H'$ on $Y$, which immediately follows from the following lemma, by the slope argument used in lemma 2.6 above. Indeed, if $\cL_X$ did not vanish identically (ou $X^{reg})$, then so would do also $\pi^*(\cL)$ and $\pi^*(\cL)\circ J$, and also its restriction to any curve $C'$ as above.

\begin{lemma}\label{factor}
 There exists a natural $\cO_Y$-linear factorisation $\cL_1:\Lambda^2\cF_1\to \pi^*(T(X,\Delta))/\cF_1$ of $\pi^*(\cL_X)\circ J:\Lambda^2\cF_1\to \pi^*(TX/F^{\rm(s)})$ through $J_X$, i.e: such that $\pi^*(\cL_X)\circ J=J_X\circ\cL_1$.
 
 More generally, $\pi^*(\cL_X)$ maps $\Lambda^2(\pi^*(T(X,\Delta))$ into $\pi^*(T(X,\Delta))$. In other words: $\pi^*(T(X,\Delta))$ is closed under the lift of the Lie bracket $\cL_X$.
\end{lemma}

\begin{proof}  The first assertion is an immediate consequence of the last one, which we now prove.

 We chose local coordinates $x=(x_1,...,x_n)$ and $y=(y_1,...,y_n)$ near $a:=\pi(b)$ and $b\in D_1$, with $D_1$ of equation $x_1=0$ the local support of $\Delta$ near $a$, so that $\pi:Y\to X$ is locally given by: $\pi(y)=(x_1:=y_1^g,x_2=y_2,...,x_n:=y_n)$ near $b$. We denote by $c$ the coefficient of $D_1$ in $\Delta$.

Local generators as $\cO_Y$-modules of $\pi^*(TX)$ (resp. $\pi^*(T(X,\Delta))$ are: $(\partial_1:=\pi^*(\frac{\partial}{\partial x_1}), \partial_2:=\pi^*(\frac{\partial}{\partial x_2}),..., \partial _n:=\pi^*(\frac{\partial}{\partial x_n}))$ and: $(y_1^{gc}.\partial _1,\partial _2,..., \partial _n))$, respectively.

Any local $\varphi\in\cO_Y$ 
can be uniquely written: 
$\varphi (y)= \sum_{k=0}^{g-1}y_1^t.\psi_t(x),$
for some holomorphic functions $(\psi_t)$. 

Let $v=\sum_{j=0}^{n}\varphi_j(y)\partial_j$ be a local section of $\pi^*(T(X,\Delta))$.
Then $v= \sum_{t=0}^{g-1}y_1^t. w_t$
with: $w_t:=\sum_{j=1}^{j=n}\psi_{j,t}(x)\partial_j$, for each $0\leq t\leq g-1$, and $\psi_{1,t}$ divisible by $x_1$ for $0\leq t\leq gc-1$. 

Then, $v$ is a section of $\pi^*(T(X,\Delta))$ translates to: for $0\leq t\leq (gc-1)$, $w_t$ is a section  
of the subsheaf $\pi^*(\cV_1):=\pi^*(T_X(-log(D_1)))$ of $\pi^*(T_X)$, generated by:
$(x_1.\partial_1, \partial_2,\dots, \partial _n)$.

\noindent Therefore we have the decomposition:
$v= \sum_{t=0}^{gc-1}y_1^k. w_t+ \sum_{p=gc}^{g-1}y_1^p. w_p$
with each $w_t\in \cV_1$, for $0\leq t \leq cg-1$.

The subsheaf $\cV_1$ of $T_X$ is stable by the Lie bracket $\pi^*(\cL_X)$, and so:
\begin{enumerate}

\item[(1)] For each $0\leq t,s\leq cg-1$ we have: $\displaystyle \pi^*(\cL_X) \big(y_1^t.w_t, y_1^s.w_s)\big)= y_1^{t+s}\pi^*(\cL_X)(w_t, w_s)$, which is a local section of $\pi^{*}\cV_1\subset \cF_1$.
\smallskip

\item[(2)] If 
we have $\max{p, q}\geq cg$, then the expression:$\displaystyle \pi^*(\cL_X) \big(y_1^p. w_t, y_1^s.
w_q)\big)= y_1^{p+q}.\pi^*(\cL_X)(\wt w_p, \wt w_q)$, which
is divisible by $y_1^{cg}$.
\end{enumerate}
\noindent The Lemma \ref{factor} is therefore proved, since for any two sections $v,v'$ of $\pi^*(T(X,\Delta))$, $\pi^*(\cL_X)(v,v')$ is a sum of terms of the preceding two forms (1) or (2) .\end{proof}
\bigskip

\noindent The sheaf $\cF^{\rm (s)}$ defines thus a foliation on the regular part of $X$,
The restriction of $\cF^{\rm (s)}$ to any curve $C$ 
which is a complete intersection of $n-1$ hyperplanes linearly equivalent to any large enough multiple of $B$ is ample, 
since this is already the case for $\cF_1$.
By  \cite{BMQ}, the leaves of $\cF^{\rm (s)}$ through any \emph{generic} point of $X$
are algebraic (since the generic curves $C$ as above 
avoid the singularities of the foliation defined by $\cF^{\rm (s)}$). The statement of \cite{BMQ} obviously holds with the very same proof in the normal case as well, since the curves we consider are contained in the regular part of $X$.
\smallskip

We thus obtain a rational fibration $f:X\dasharrow Z$, such that for generic $x\in X$, the kernel of the differential 
$df_x$ is equal to $\cF^{\rm (s)}_x$. The idea to finish the proof is that,
since $(K_X+ \Delta)$ is pseudo-effective\footnote{And since $(X,\Delta)$ is log-canonical. This the place where this hypothesis is used.}, the relative canonical bundle of $f$ is pseudo-effective on any `neat' model of $f$,
which contradicts the positivity of the degree of $\cF^{\rm (s)}$ when restricted to a generic curve $C$.
The quotient sheaf $Q_{f,\Delta}$ of $\pi^*(\Omega^1(X,\Delta))$ we have considered has however as kernel, not $f^*(\Omega^1_Z)$, but its
saturation in $\pi^*(\Omega^1(X,\Delta))$. The difference for the degree computed on $C'$ is however, after proposition \ref{deg} above, interpreted geometrically as coming from the orbifold divisor $\Delta$ and the multiple fibres\footnote{They thus play a crucial role even when $\Delta=0$.} of any `neat' model of $f$. The needed refinement of the pseudo-effectivity of the relative canonical bundles turns out to be essentially the ones given either in \cite{Ca04}, theorem 11.3; it can be equally 
extracted from \cite{K} or \cite{BP}.

We introduce some notations and definitions: given a surjective map $\varphi: M\to N$ between two projective manifolds
$M$ and $N$, we denote by $D_N(\varphi)$ the set
$$D_N(\varphi):= \{y\in N\vert \varphi^*(y) \hbox{ is not smooth} \}.$$
Let $D_M(\varphi):= \varphi^{-1}\big(D_N(\varphi)\big)$ be the inverse image of $D_N(\varphi)$. We also consider 
a divisor $\Delta$ on $M$; in this context, we recall the following notion.

\begin{defn}\label{dneat} We say that the map $\varphi$ is $\Delta$-neat if the following requirements are fulfilled.

\begin{enumerate}

\item[(a)] The set $D_N(\varphi)$ is a (possibly empty) divisor.
\smallskip

\item[(b)] The divisors $D_N(\varphi)$ and $\Delta+ D_M(\varphi)$ have normal crossings.
\smallskip

\item[(c)] No component of $\Delta$ is $\varphi$-exceptional.
\smallskip

\end{enumerate}

\end{defn}

Starting from our initial log-canonical $(X,\Delta)$, we can thus take a log-resolution $g:X'\to X$ such that $X'$ is smooth, and a smooth orbifold pair $(X',\Delta')$ with $f':X'\to Z'$ holomorphic and birationally equivalent to $f$ via a modification $v:Z'\to Z$, $Z'$ smooth, such that: $g_*(\Delta')=\Delta$, $K_{X'}+\Delta'=g^*(K_X+\Delta)+E$, with $E$ $g$-exceptional, and such that, moreover, $f':X'\to Z'$ is $\Delta'$-neat. Because our curves $C=g_*(C')$, with $C'\subset Y$ a Mehta-Ramanathan curve for $H'=\pi^*(B)$ do not meet the indeterminacy locus of $g^{-1}:X\dasharrow X'$, we still know that $\cF^{(s)}$ is ample on $C'$ (identified with its isomorphic strict transform in $X'$). We can and shall now thus argue as if $X=X'$.

We decompose the (new) divisor $\Delta$ (which lies on the new $X'=X)$ as follows:
$$\Delta= \Delta^{\rm vert}+ \Delta^{\rm hor}$$
so that each component of the support of $\Delta^{\rm vert}$ maps via $f$ onto some divisor of $Z$, while the restriction of $f$ to any component of 
the support of $\Delta^{\rm hor}$ is surjective. By the $\Delta$-\emph{neat} condition, only these possibilities can occur.

Notice that, since the 'new' $\Delta(=\Delta')$ on $X'$ differs from the lift of the initial $\Delta$ on $X$ only by components contained in the exceptional locus of $g$, whch does not meet $C=\pi(C')$, we can apply the local computations of proposition \ref{deg} as if we actually had $X'=X$ and $\Delta'=\Delta$.

We now use the notations introduced before the proof of proposition \ref{deg}. By assumption, the quotient $Q_{f,\Delta)}$ has an ample dual over $C'$. In particular, it has negative degree on $C'$. 

By proposition \ref{deg}, this degree is given by: $$(\frac{1}{g.m^{n-1}}).deg_{C'}(Q_{f,\Delta})=[K_{X/Z}+\Delta-D(f,\Delta)].H^{n-1},$$

where $D(f,\Delta):=[\sum_{k=1}^{k=r} (t_{F_k}-\frac{1}{m_{\Delta}(F_k)}).F_k]$. Recall that the sum in $D(f,\Delta)$ bears on the finitely many irreducible divisors $F_k$ of $X$ which are either components of $\Delta^{vert}$, or mapped by $f$ onto divisors of $Z$ with multiplicity $t_k\geq 2$. Also, $m_{\Delta}(F_k)\geq 1$ is the $\Delta$-multiplicity of $F_k$.

 \
 
 Since we assumed $K_X+\Delta$ (and thus also $K_{X'}+\Delta'$) to be pseudo-effective, this negativity contradicts the following result (which thus ends the proof of theorem \ref{tgsp}):

\begin{theorem}\label{corcb} Let $(X, \Delta)$ be a smooth orbifold pair, and let $f:X\to Z$ be a $\Delta$-neat fibration. 
If $K_{X_y}+ \Delta_{\vert X_y}$ is pseudo-effective on the generic fibre $X_y$ of $f$, the $\Q-$bundle $K_{X/Y}+ \Delta-D(f,\Delta)=K_{X/Y}+ \Delta^{hor}-D(f,0)$ is then pseudo-effective, too.

\end{theorem}

\noindent {\bf Proof:} The result above is an easy consequence of  \cite{Ca04}, Theorem 4.13, applied to $D:=m.\Delta^{hor}$. This result indeed asserts that $f_*(m(K_{X/Z}+\Delta^{hor}))$ is weakly positive, and $m(K_{X/Z}+\Delta^{hor}))$ is thus pseudo-effective.

The proofs of  \cite{Ca04}, lemma 4.17 and lemma 4.18, applied with $H^{vert}=0$, now shows that this last conclusion is preserved when we substract from $m(K_{X/Z}+\Delta^{hor})$ not only $g^*(\Delta (g,H))$, as stated there, but even $D(f,0)$. Observe indeed that we can write, for any component $F_k$, $g^*(g(F_k))= t_k.F_k$ along its generic point, so the calculations at the end of the proof of \cite{Ca04}, 4.18 give the assertion (just ignore the 7 last lines of the proof of 4.18). The equalities: $t-\frac{1}{m}=(t-1)+(1-\frac{1}{m})=t.(1-\frac{1}{t})+(1-\frac{1}{m})$ finally imply that $D(f,\Delta)=\Delta^{vert}+D(f,0)$. Notice that the multiplicities of $\Delta^{vert}$ here do not play any role, and can be chosen to be rational, not necessarily integral.

 In the special case where $\Delta=0$, it is stated in \cite{K}. The general case can be obtained as well from  \cite{BP}, by an adaptation similar to the one above from \cite{Ca04}. \end{proof}

\subsection{An alternative approach}

\

\

In a forthcoming text \cite{CP14}, we will provide a different proof of Theorem \ref{tgsp} by using 
differential-geometric techniques. Our arguments are based on the existence of K\"ahler metrics with conic singularities and prescribed Ricci curvature; the precise statement is as follows.

\begin{theorem}\label{m-tgsp} Let $(X, \Delta)$ be an orbifold pair, whose 
canonical bundle $K_X+\Delta$ is pseudo-effective. Let $L$ be a line bundle, such that 
$$H^0\big(Y, \otimes ^m\pi^\star T(X, \Delta)\otimes L\big)\neq 0$$
for some $m\geq 1$. Then we have $L\cdot \pi^\star H^{n-1}\geq 0$.
\end{theorem}
\medskip

\noindent It follows that 
the restriction of $\pi^\star\big(\Omega^1(X, \Delta)\big)$ to a 
generic complete intersection is nef, in the sense of algebraic geometry. 
\smallskip
 
\noindent We will present here the main techniques used for proof of Theorem \ref{m-tgsp}.
We denote by $Y_0$ a complex manifold, which is not necessarily compact. The metric $\omega$ with respect to which the next computations are performed is assumed to be K\"ahler. The bundle $L$ 
is endowed with a hermitian metric $h$, and we denote by $\displaystyle 
\tr_\omega\big(\Theta_h(L)\big)$
the trace of the curvature of $L$ with respect to $\omega$. The following Bochner-type formula 
is classical, cf. \cite{BY}. 

\begin{lemma}
\label{lem41}
Let $u$ be a $L$-valued tensor of $(m, 0)$-type on $Y_0$, with compact support. Then we have
\begin{align*}
\int_{Y_0}|\overline\partial(\# u)|^2dV_\omega= & \int_{Y_0}|\overline\partial u|^2dV_\omega
+ \\
+ & \int_{Y_0}\langle \cR(u), u\rangle dV_\omega+ \int_{Y_0}|u|^2\tr_\omega\big(\Theta_h(L)\big)dV_\omega,
\end{align*}
where $\cR$ is an order zero operator, defined as follows. We write 
\begin{equation}
\label{23}
u= \sum_{I}u^I{\partial\over \partial z^{\otimes I}}\otimes e_L
\nonumber
\end{equation}
and then we have

\begin{align*}
\cR(u):= & \sum_{I, p, l}u^IR_{i_p\overline l}
{\partial\over \partial z^{i_1}}\otimes \cdots \otimes
{\partial\over \partial z^{i_{p-1}}}\otimes
{\partial\over \partial z^{l}}\otimes
{\partial\over \partial z^{i_{p+1}}}\otimes \cdots \otimes
{\partial\over \partial z^{i_r}}\otimes e_L
\end{align*}
In the expression above, we use the notation
$$R_{j\overline i}= \sum_p R_{p\overline p j\overline i}$$
for the coefficients of the Ricci tensor (here all the quantities 
are expressed with respect to some geodesic coordinates), as well as 
$I= (i_1,\ldots, i_r)$.  
\end{lemma}

\medskip

\noindent We will use the preceding result as follows. The manifold $Y_0\subset Y$ corresponds to 
the smooth, non-ramified cover of $X\setminus \Supp (\Delta)$.  Let $f$ be any smooth function on 
$Y$; we denote by $h_0$ a reference metric on $L$, and let 
$$h:= \exp(-f)h_0$$
be the twisting of the reference metric on $L$ with the function $f$. The last term 
in the equality of the previous lemma becomes
$$\tr_\omega\big(\Theta_{h_0}(L)\big)+ \Delta_\omega(f)$$
where $\Delta_\omega$ above is the Laplacian operator associated to $\omega$.
As a consequence, the term corresponding to $L$ in Lemma \ref{lem41} becomes
$$\int_{Y_0}\big(\tr_\omega\big(\Theta_h(L)\big)+  \Delta_\omega(f)\big)|u|^2\exp(-f)dV_\omega.
$$
\medskip

Coming back to our problem, if $K_X$ is ample and $\Delta= 0$, then the proof of Theorem \ref{m-tgsp} is as follows. Let 
$\omega\in c_1(H)$ be a representative whose Ricci curvature is definite negative. Such a metric exists as a consequence of the ampleness of $K_X$, thanks to S.-T. Yau theorem, cf. \cite{Yau}. 
We choose $h_0$ in an arbitrary manner, and let 
$$f= \log |u|^2,$$
where $u$ is the $L$-twisted tensor of ($m, 0$)-type given by hypothesis
(and the norm above is induced by $\omega$ and $h_0$). With this choice, we have
$$\int_{X}\big(\tr_\omega\big(\Theta_h(L)\big)+  \Delta_\omega(f)\big)|u|^2\exp(-f)dV_\omega=
L\cdot H^{n-1}
$$
since the $|u|^2$ is cancelled, and the integral of $\Delta(f)$ with respect to 
$dV_\omega$
is equal to zero. The term 
$$\int_{X}\frac{\langle \cR(u), u\rangle }{|u|^2} dV_\omega$$
is negative, by the properties of the Ricci curvature of the metric $\omega$, and $\overline\partial u= 0$, since $u$ is holomorphic. Thus we infer the result.

The general case is much more involved than this, but the techniques needed to carry it on are well-understood. If $\Delta\neq 0$ then we have to use the cut-off procedure and the conic singularities metrics as in 
\cite{CGP}. If $K_X+ \Delta$ is only pseudo-effective rather than ample, then we approximate it 
with a big line bundle, and we use Kodaira lemma i.e. we have 
$$K_X+ \Delta+ \varepsilon H\equiv A_\varepsilon+ E_\varepsilon$$
for each $\varepsilon> 0$, where $A_\varepsilon$ is ample and $E_\varepsilon$
is effective. The bundle $E_\varepsilon$ will induce a further degeneracy in the volume 
element while solving the Monge-Amp\`ere equation. However, the 
estimates we have at our disposal in this framework are solid enough to enable us to 
argue by approximation. As we have already mentioned, the details will appear shortly in 
\cite{CP14}.\qed

%%%%%%%%%%%%%%%%%%%%%%%%%%%%%%%%%%%%%%%%%%%%%%%%%%%%%%%%%%%%%%%%%%%%%%%%%%%%%%%%%%%%%%%%%%%%%%%%

\section{Birational stability of the orbifold cotangent bundles}

\noindent We now give a consequence of Theorem \ref{tgsp} (which was its original motivation). 
For similar results, we refer to \cite{Ca09}, and to \cite{CGP}, where transcendental methods are used.
 
 \begin{corollary}\label{c1} Let $(X,\Delta)$ be a log-canonical orbifold pair, with $X$ normal projective and $K_X+\Delta$ pseudo-effective. Let $H$ be any ample line bundle on $X$. Let $\pi:Y\to X$ be a cyclic cover of group $G$ associated to $(X,\Delta)$. Let $H':=\pi^*(H)$. Let $\cF'$ be a rank-one\footnote{The results hold in fact for $det(\cF')$ if $rk(\cF')>1$, with the same proof.} coherent sheaf on $Y$, together with an inclusion $\cF'\subset \otimes^m(\pi^*(\Omega^1(X,\Delta)))$. 
 
 Assume that $(K_X+\Delta).H^{n-1}=0$. Then: 

1. $\cF'.(H')^{n-1}\leq 0$.

2. $h^0(Y,\cF')\leq 1$.

3. More generally\footnote{ The assertions above remain true after lifting $\cF'$ and $H'$ by $\psi^*$, if $\psi:Z\to Y$ is any surjective holomorphic map from an irreducible normal complex space $Z$ to $Y$.}, the evaluation map at a generic point $y\in Y$: $$e_y: H^0(Y,\otimes^m(\pi^*(\Omega^1(X,\Delta))))\to \otimes^m(\pi^*(\Omega^1(X,\Delta)))_y$$ is injective if $(K_X+\Delta).H^{n-1}=0$. \footnote{This is a version of the fact that holomorphic tensors are `parallel' in this situation, a fact proved when $\Delta=0$ in the smooth K\"ahler case by S.T.Yau using Ricci-flat K\"ahler metrics and Bochner formula, and the later in the projective case by Y. Miyaoka using his generic semi-positive theorem just as above}.

Assume that $\Delta=D+\Delta'$, for some $\Bbb Q$-effective, non-zero $D,\Delta'$, and that $\cF'\subset \otimes^m(\pi^*(\Omega^1(X,\Delta')))$. Then:

1'. $\cF'.(H')^{n-1}< 0$.

2'. $h^0(Y,\cF')=0$.

 \end{corollary}
 
 {\bf Proof:} The assertion 2 follows obviously from assertion 1, which we now prove. Let $C'\subset Y$ be a Mehta-Ramanathan curve for $H':=\pi^*(H)$, and $C:=\pi_*(C')$.
 
 Assume first that $(K_X+D).H^{n-1}=0$. Let $Q'$ be the quotient of $\otimes^m(\pi^*(\Omega^1(X,\Delta)))$ by $\cF'$ over $Y$. By theorem \ref{tgsp}, and its corollary \ref{cgsp}, $\det(Q'_{C'})\geq 0$ (since $K_X+\Delta$ is assumed to be pseudo-effective). But $det(Q').C'=-\cF'.C'$, since $(K_X+\Delta).C=0$. Hence the claim.
 
 In the second case, where $(K_X+\Delta').H^{n-1}<0$, the inclusion:
 
  $\cF'\subset \otimes^m(\pi^*(\Omega^1(X,\Delta')))\subset \otimes^m(\pi^*(\Omega^1(X,\Delta)))$ permits to deduce the last two assertions from the preceding ones, since $C$ meets the support of $\Delta'$. $\square$
 
 \begin{rem} The preceding corollary \ref{c1} applies if $(X,\Delta)$ is the image of some smooth orbifold pair $(X",\Delta")$ by a rational birational map $\mu:X"\to X$ whose inverse does not contract any divisor, and with $K_{X"}+\Delta"$ pseudo-effective.
 
 Under the `Abundance conjecture', if $\kappa(X",\Delta")=0$, the property $(K_X+\Delta).H^{n-1}=0$ will be satisfied on any Log-minimal model of $(X",\Delta")$). 
 \end{rem}

  \begin{rem}The second case $(X,\Delta')$ of the preceding corollary arises, for example, when $(X,\Delta')$ is Fano (i.e: has $-(K_X+\Delta')$ is ample), by adding to $\Delta'$ some $D=\frac{1}{N}.E$, where $E$ is a generic member of the linear system $-N.(K_X+\Delta')$.\end{rem}
 
\begin{rem} In these cases, using the invariant $\kappa^{++}$ introduced in \cite{Ca09}, the corollary \ref{c1} shows in particular that $\kappa^{++}(X",\Delta")=0$ (resp. $-\infty$) if $(K_X+\Delta).H^{n-1}=0$ (resp. $(K_X+\Delta').H^{n-1}<0$). \end{rem}

\section{A criterion for orbifold pairs of general type}

The following result\footnote{We give the statement only in its `pure-logarithmic' version, which involves no `orbifold' consideration. But it holds, and the proof given below adapts immediately for general log-canonical pairs $(X,\Delta)$, which are needed in the proof of the `purely logarithmic' case already. In the general case, the assumption is that the inverse image $\pi^*(L)$ of a big line bundle $L$ on $X$ injects in $\otimes^m(\pi^*(\Omega^1(X,\Delta))$ for some cyclic cover $\pi:Y\to X$ associated to $\Delta$.}  was conjectured by 
E. Viehweg in \cite{VZ}.

\begin{theorem}\label{strict} Let $X$ is a projective manifold, and 
$D= \sum_jD_j$ a reduced divisor, such that $(X, D)$ is a smooth `purely-logarithmic' orbifold pair. We assume
the existence of a big line bundle $L$ on $X$, together with an injective sheaf map
$$0\to \cO(L)\to \otimes^m\Omega^1(X, D)\leqno(2)$$
for some integer $m\geq 1$.  Then $K_X+ D$ is big.
\end{theorem}

\begin{rem} 

{\rm We mention some complements and extensions which can be obtained by similar arguments:

0. When $L$ is not assumed to be big, $K_X+\Delta$ need not be pseudo-effective, even if $L$ is effective, in general (consider $X=\Bbb P^d\times Z$, $Z$ of general type, $f$ the second projection, $\Delta=0$, and $L:=f^*(K_Z)$, with $K_Z$ effective, but not big). The second step of the argument below thus requires the bigness of $L$.

1. The argument proving theorem \ref{strict} can be extended with minor changes, to show that if $L$ and $K_X+\Delta$ are supposed to be pseudo-effective, then $\nu(K_X+\Delta)\geq \nu(L)$, where $\nu$ stands for the numerical dimension .

2. Remark also that, since the tensor product of two line bundles, one big and the other pseudo-effective, is big, the conclusion of the theorem were obvious if one could prove that the quotients of $\Omega^1(X,\Delta)$ have a pseudo-effective (instead of gsp) determinant, under the hypothesis of theorem \ref{tgsp}. This stronger property has been shown in \cite{CPe} when $\Delta=0$ if $X$ is smooth and projective.

3. The initial part of the proof  of theorem \ref{strict} actually applies to give, with an additional nefness assumption, a `distributional' version of Theorem \ref{strict} (see the beginning of the proof of its step 1) :

\begin{theorem}\label{strict'} Let $(X,D)$ be a pair consisting of a complex smooth projective manifold $X$, equipped with a normal crossing (reduced) divisor $D$. Let $Q$ be a torsion free quotient of $\Omega^1_X(Log D)$\footnote{Thus seen as the dual of a saturated subsheaf $\cF$ of $T_X(Log D)$, and $det(Q)=K_{\cF}$.}. Assume that $det(Q)$ is nef, and that there exists an injective sheaf map $L\to \otimes^mQ$ for some $m>0$. Then $det(Q)$ is big.
\end{theorem} 

\

\noindent It might be possible that the result holds more generally if $det(Q)$ is pseudo-effective, but additional arguments concerning Log-minimal models or Zariski decomposition of $det(Q)$ were then needed. On the other hand, the statement does not hold if $det(Q)$ is not pseudo-effective, as the following example shows.

\begin{ex} {\rm 
Let $S$ be a minimal surface of general type, such that we have $13c_1^2> 9c_2$, where 
$c_1$ and $c_2$ are the first and the second Chern class of $S$, respectively. 
Let $X:= {\mathbb P}(T_S)$ be the projectivization of tangent bundle of $S$, and let 
$\Lambda \subset T_X$ be the sub-bundle described by the following relation

$$\Lambda_{(x, [v])}:= \{\xi\in T_{X, (x, [v])}\text{ such that } d\pi (\xi)\in {\mathbb C}v\}$$
where $v\in T_{S, x}$ is a non-zero tangent vector, and $\pi: X\to S$ is the projection map. We remark that $\Lambda$ is not integrable; 
geometrically, it corresponds to the directions of $X$ corresponding to liftings of discs tangent to $S$. We denote by $Q:= \Lambda^\star$,
the dual of $\Lambda$. Then for any ample line bundle $A$ on $X$ there exists an integer $m$ such that we have 
$$H^0\big(X, S^m Q\otimes A^{-1}\big)\neq 0 \leqno{(\dagger)}$$
but the determinant of $Q$ is not even psef, given that its restriction to the fibers of $\pi$ is 
equal to $\mathcal O(-1)$. For a proof of $(\dagger)$ and much more we refer to the article 
\cite{D97}.

\noindent Moreover, we remark that we have
$$H^0\big(X, S^m \Omega^1_X\otimes A^{-1}\big)= 0$$
for any $m\geq 1$, by the same arguments. This may look odd, given $(\dagger)$, yet it is true.}
\end{ex}

\

5.  Actually, the techniques we use in the proof of Theorem \ref{strict} permit to characterize the bigness
of $K_X+ \Delta$, at least in the ``purely logarithmic" case (i.e. $m_j= \infty$). 
We denote by ${E}_{k, m}^{GG}(\Omega^1_{X, \Delta})$ the bundle of logarithmic jet differentials 
or order $k$ and degree $m$. Let $l$ be a positive integer. Then Theorem \ref{strict} admits the following extension and reciprocal version.

\begin{theorem} \label{higher}
The bundle $K_X+ \Delta$ is big if and only if there exist a couple of positive integers $k, m$ together with an injective sheaf map $\displaystyle \cO(L)\to \otimes ^lE_{k, m}^{GG}\Omega^1_{X, \Delta}$ where $L$ is an ample line bundle.
\end{theorem} 

\noindent The ``only if"
part follows from the techniques developed in the article \cite{D} by J.-P. Demailly, and 
the ``if" part is a consequence of Theorem \ref{strict}, as follows.

The logarithmic ``Green-Griffiths" bundle of jet differentials of order $m$ and degree $k$
admits the filtration whose successive quotients are given by
$$S^{m_1}\Omega^1_{X, \Delta}\otimes \dots\otimes S^{m_k}\Omega^1_{X, \Delta}$$
where $\displaystyle (m_j)_{j=1\dots k}$ are positive integers, such that 
$$m_1+2m_2+\dots+ km_k= m.$$
Then we infer that there exists some integer $q$ such that 
$$H^0\big(X, \otimes^{q}\Omega^1_{X, \Delta}\otimes L^{-1}\big)\neq 0$$
and therefore Theorem \ref{higher} is a direct consequence of \ref{strict}.

Also, we mention here that it might be possible to develop the theory of jet differentials in the 
context of general orbifold pairs, and prove a similar result. }

\end{rem}
\smallskip

\begin{proof} (of Theorem 4.1) Let $A$ be a very ample line bundle on $X$ having a section whose zero set $Z$ is smooth and such that $D\cup Z$ is of normal crossings, $A$ being sufficiently multiplied, so that $K_X+D+\frac{1}{2}.A$ is pseudo-effective. Consider the orbifold pair $(X,D+t.Z)$, for $t\geq 0$ rational. The proof consists of two steps: 

\

{\bf Step 1.} $K_X+D+t.Z$ is big if $K_X+D+t.A$ is pseudo-effective, with $0\leq t <1$. We shall prove this after proving the second step.

\

{\bf Step 2.} $K_X+D$ is pseudo-effective. We prove this step 2 now, assuming step 1. Assume, by contradiction, that $K_X+D$ is not pseudo-effective.  Let $\frac{1}{2}\geq t_0>0$ be the smallest of the real numbers $t$ such that $K_X+D+t.Z$ is pseudo-effective. By \cite{BCHM}, $t_0\in \Bbb Q$. By the first step, $K_X+D+t_0.Z$ is big. But this implies that $K_X+D+(t_0-\varepsilon).Z$ is pseudo-effective for some $\varepsilon>0$, contradicting the definition of $t_0$.

\

{\bf Proof of step 1:} We first illustrate the idea in the special case where $K:=K_t:=K_X+D+t.A$ is nef. Let $a>0$ be such that $L\geq a.A$ (i.e: such that the difference $L-a.A$ is $\Bbb Q$-effective), and let $c=c(n,m)>0$ be such that $det(\Omega^1(X,D+t.Z))=c.K_t=c.(K_X+D+t.A)$. We then have (using the fact that $\Omega^1(X,D+t.Z)$ is gsp, thus as well as its tensor powers, and the Khovanskii-Teissier inequalities for the third and first inequalities, respectively): $$a.(A^n)^{\frac{1}{n}}.(K_t^n)^{\frac{n-1}{n}}\leq a.A.K_t^{n-1}\leq L.K_t^{n-1}\leq (c.K_t).K_t^{n-1}=c.K_t^n,$$
from which we deduce that $Vol(K_t)=K_t^n\geq (\frac{a}{c})^n.vol(A)>0$ (in order to divide both sides of the inequality above by $(K_t^n)^{\frac{n-1}{n}}$, which might, a priori, be zero, one just needs to apply the inequality to $t+\varepsilon, \varepsilon >0$ rational, and let $\varepsilon$ tend to zero). This implies that $K_t$ is big.

\

Notice that this special nef case works exactly in the same way if $\Omega^1(X,D+t.Z)$ is replaced by any of its torsionfree quotients $Q$, to give Theorem \ref{strict'}.

\

We then reduce to the case when $K_t:=K_X+D+t.A$ is nef, assuming it to be pseudo-effective, by using \cite{BCHM}. We shall give two proofs of step 1. We consider in both proofs the sequence of klt orbifold divisors $D_{t,k}:=(1-\frac{1}{k.N}).D+t.Z+\frac{1}{kM}.(M.(A+\frac{1}{N}. D))\sim D+(t+\frac{1}{k}).A$, where $k>0$ is an integer, and $N,M$ are chosen such that $M.(A+\frac{1}{N}. D)$ is very ample, and has a section with zero locus $Z'$ such that $D\cup Z\cup Z'$ is of normal crossings. The divisor $D+(t+\frac{1}{k}).A$ is big, and [BCHM06] applies. Here $t\geq 0$ is fixed and $k$ varies.

\

{\bf First proof:} %Raising $L$ to some suitable power,we can, and shall, assume that $L$ is `very big' (i.e: its linear system gives a birational map on $X$).
By \cite{BCHM}, there exists a composition $\mu:X\dasharrow X'$ of divisorial contractions and flips such that $(X', D_{t,k}':=\mu_*(D_{t,k}))$ is l.c, and $X'$ is $\Bbb Q$-factorial, with $K':=K_{X'}+D'_{t,k}$ nef. Let $L':=\mu_*(L)$: this is a big $\Bbb Q$-Cartier rank one coherent sheaf (well-defined since $\mu^{-1}$ does not contract any divisor). We have for the same reason a natural injection of sheaves $L'\to \mu_*(\otimes^m\Omega^1(X,D))\to (\otimes^m\Omega^1(X',D'))$\footnote{Recall that the cotangent sheaf has been defined by extension from any suitable Zariski open subset with codimension two complement.} if $D':=\mu_*(D)$. For any $t\geq 0, k>0$, we also get (after lifting to a suitable cyclic cover of $X'$) an injection of sheaves $\otimes^m\Omega^1(X',D')\to \otimes^m\Omega^1(X',D'_{t,k})$.

Let $\nu:X"\to X$ be a birational morphism such that $\rho:=\mu\circ \nu:X"\to X'$ is regular, and such that the indeterminacy locus of $\nu^{-1}:X\dasharrow X"$ is included in the indeterminacy locus of $\mu:X\dasharrow X'$. Let $L":=\nu^*(L), A":=\nu^*(A)$. There exists a Zariski-open subset $U'$ of $X'$ with codimension two or more complement which is isomorphic via $\mu$ (resp. $\rho$) to its inverse image $U\subset X$ (resp. $U"\subset X"$). Since $K_{X'}+D'_{t,k}$ is nef, $K'_{\varepsilon}:=K_{X'}+D_{t,k}'+\varepsilon A'$ is ample for any $\varepsilon>0$, rational, and $A'$ an ample line bundle on $X'$. We now consider a curve $C'_{\varepsilon}$ which is a complete intersection of $(n-1)$ generic members of $N'.K'_{\varepsilon}$, $N'$ sufficiently big, in such a way that $C'\subset U'$. Let $\Gamma_{\varepsilon}:=\frac{1}{(N')^{n-1}}.C'$, and $\Gamma_{\varepsilon}"$ be its inverse image in $X"$ by $\rho^*$. 

We thus have: $L'.(K'_{\varepsilon})^{n-1}=L'.\Gamma_{\varepsilon}=L".\Gamma_{\varepsilon}"=L".(\rho(K'_{\varepsilon}))^{n-1}$, for any $\varepsilon>0$.

Now, we can choose a rational effective divisor $\Delta'\cong \varepsilon A'$ on $X'$ such that the pair $(X',D'_{\varepsilon}=D_{t,k}'+\Delta')$ is klt. Since $\Omega^1(X',D'_{\varepsilon})$ is then gsp, and $K_t'$ is nef, we get, by letting $\varepsilon \to 0^+$, since then $K'_{\varepsilon}\to K'_t$: $$a.A".((\rho^*(K_t')^{n-1})\leq L".(\rho^*(K_t')^{n-1})=L'.(K_t')^{n-1}\leq c.K_t'^{n}$$

The crucial point here is that the constants $a$ and $c$ are independent on $t,k$ and $\varepsilon$.

The rest of the proof is then just as in the case where $K_t$ is nef, letting $k\to+\infty$, using the continuity of the volume, and the equality: $vol(K_X+D+(t+\frac{1}{k}).A)=vol(K_{X'}+D'+(t+\frac{1}{k}).A^*)=K'^{n}$ if $A^*:=\mu_*(A)$. This finishes the first proof.

\

{\bf Second proof:} Fix $t,k$ as above. We work with the orbifold divisor $D_{t,k}:=D+\frac{1}{N}.Z$, where $Z$ is a generic member of the linear system $\vert M. A\vert$, such that $M=N.(t+\frac{1}{k}),$ and $N,M$ are sufficiently big integers.

Because $K_X+D+t.A$ is nef, $D+(t+\frac{1}{k}).A$ is big, and $K_X+D_{t,k})$ can be written as a klt divisor, the associated canonical algebra $R_{t,k}$ associated to $K:=K+D_{t,k}$ is finitely generated, after \cite{BCHM}. There thus exists a Zariski decomposition for $K_{t,k}$, that is: a modification $p:X'\to X$ with $X'$ smooth, $p^*(D_{t,k}\cup Exc(p))$ is of simple normal crossings such that $p^*(K)=P+N$, where $P$ is big, without base points, with the same volume $Vol(P)=P^n=vol(K)$, $N$ is effective, and $N.P^{n-1}=0$.

The modification $p$ is a suitable sequence of blow-ups with smooth centers making the ideal $\cI\subset \cO_X$ locally generated by the vanishing loci of a set of generators of the algebra $R_{t,k}$. By \cite{Kol}, we can, moreover, chose this sequence of blow-ups in such a way that, additionally, the support of $K_{X'/X}$ is contained in the inverse image of the cosupport of the preceding ideal $\cI$, where all sections of the generators of $R_{t,k}$ vanish. This property implies that $F.P^{n-1}=0$ for each irreducible component $F$ of $Exc(p)$.

Because $(X,D)$ is log-canonical and $D$ is reduced, we have: $p^*(K_X+D)+E=K_{X'/X}+\ol D+E'$, with $E\cup E'\subset Exc(p)$, $\ol D$ the strict transform of $D$, and $E'$ reduced such that $\ol D+E'\subset p^*(D)$. Thus $p^*$ induces an injection $\Omega^1(X,D)\to \Omega^1(X',D')$, where $D'$ is the reduced part of $p^*(D)$.

The injection
$L\to \otimes^m(\Omega^1(X, D))$ thus also lifts to:
$$p^*(L)\to \otimes^m (\Omega^1(X', \ol D+ \frac{1}{N}.p^*(V))),$$ 
such that $(X', \ol D+ \frac{1}{N}.p^*(V))$ is also log-canonical, by the generic choice of $V$, which permits to impose that $V$ does not contain any component of the cosupport of the ideal $\cI$.

The injection $p^*:\Omega^1(X,D)\to \Omega^1(X',D')$ shows that $K_X'+t.p^*(A)$ is pseudo-effective, and so $\Omega^1(X',D'+(t+\frac{1}{k}).p^*(A)$ is generically semi-positive, so that, putting $K':=K_{X'}+ D' +(t+\frac{1}{k}). p^*(A)$, we get the first inequality below:

$$p^*(L)\cdot P^{n-1}\leq c. K'\cdot P^{n-1}=c.p^*(K).P^{n-1}=c.(P+N).P^{n-1}=c.P^n,$$

the second equality comes from the fact that $F.P^{n-1}=0$ for each component $F$ of $Exc(p)$.

We can now conclude as when $K$ is nef, since $P^n=Vol(K)$. \end{proof}

 From \cite{VZ} (see \cite{Kebekus}, which, among many other things, surveys in a detailed way the problem, the notions involved, and the known special cases) we get:

\begin{corollary} Let $f:X\to B$ be a projective submersion between quasi-projective manifolds $X,B$. Assume that the fibres are (connected) canonically polarized manifolds. If the variation $Var(f)$ of the family is maximal (i.e. equal to $dim(B)$), then $B$ is of log-general type  (i.e: $K_{\bar B}+D$ is big, for any smooth projective compactification $\bar B$ of $B$ with complement $D:=\bar B-B$ a divisor of simple normal crossings on $\bar B)$.
\end{corollary}

The two main cases known before were \cite{KK} (the three-dimensional case), and \cite{Pat} (the case where $B$ is either compact, or admits a non-uniruled compactification). The solution of \cite{KK} rests on the knowledge of the abundance conjecture in dimension $3$, while the solution of \cite{Pat} rests on the main result of \cite{CPe}. The surface case is treated by different methods in \cite{KK'}.

\begin{rem}\label{ric} A stronger statement, called the `isotriviality conjecture', stated in \cite{Ca09}, asserts that a family of canonically polarized manifolds $f:X\to B$ as above is isotrivial if $B$ is `special', an algebro-geometric notion introduced in \cite{Ca07}. Specialness roughly means `opposite' to (Log)-general type. This stronger statement is actually the exact higher-dimensional formulation of the original conjecture of Shafarevich (proved by A. Parshin in \cite{pars}), once `special' quasi-projective manifolds are seen as the higher-dimensional versions of non-hyperbolic quasi-projective curves.The methods of the present paper might permit to attack this stronger conjecture by using the refinement of \cite{VZ} given in \cite{Jabkeb}, asserting that the `Viehweg-Zuo sheaf' comes from the moduli stack. This conjecture is established in \cite{JK} in dimensions at most $3$.
\end{rem}

\noindent {\bf Acknowledgements.} It is our 
pleasure to thank Mircea Musta\c t\u a who patiently explained to  
us many relevant facts concerning the 
algorithm of desingularization of algebraic varieties used in the second proof of the first step in the proof of theorem \ref{strict}.

\noindent {\bf Note added in proof.} The `isotriviality conjecture' mentioned in the above remark \ref{ric} has inbetween been proved by B. Taji in \cite{T} using the approach suggested there.


\begin{thebibliography}{}








\bibitem[B-P]{BP} B. Berndtsson-M. P\u aun.
{Quantitative extensions of pluricanonical forms and closed positive currents.} Nagoya Math. J. 205 (2012), 25-65.
\smallskip

\bibitem[BCHM 06]{BCHM} C. Birkar-P. Cascini-C. Hacon-J. McKernan. Existence of minimal models for varieties of log-general type. arXiv 0610.203
\smallskip

\bibitem[BY53]{BY} Bochner, S. and Yano, K. Curvature and Betti numbers. Annals of Mathematical Studies 1953
\smallskip

\bibitem[BDPP 04]{BDPP} S. Boucksom-.JP. Demailly-.M.P\u aun-T. Peternell. The pseudo-effective cone of a compact K\" ahler manifold and varieties of negative Kodaira dimension. arXiv 0405285
\smallskip

\bibitem[B-McQ 01]{BMQ} Bogomolov-McQuillan. Rational curves on foliated varieties. IHES preprint IHES/M/01/07.F\'evrier 2001.
\smallskip

\bibitem[Bo01]{Bo} JB. Bost. Algebraic leaves of algebraic foliations over number fields. Publ. Math. Inst. Hautes Etudes Sci. 93 (2001), 161Ð221.


\bibitem[Ca 04]{Ca04}F. Campana. Orbifolds, special varieties and classification theory. Ann. Inst. Fourier 54 (2004), 499-665.
\smallskip

\bibitem[Ca 07]{Ca07}F. Campana. Orbifoldes g\'eom\'etriques sp\'eciales et classification bim\'eromorphe des vari\'et\'es K\" ahl\'eriennes compactes. JIMJ 10 (2011), 809-934.
\smallskip

\bibitem[Ca09]{Ca09} F.Campana. Special orbifolds and birational classification: a survey. arXiv 1001.3763.
\smallskip

\bibitem[C-G-P 11]{CGP}F. Campana-H. Guenancia-M. P\u aun. Metrics with cone singularities along normal crossing divisors and holomorphic tensor fields. arXiv 1104.4879
\smallskip

\bibitem[CP]{CPe}F. Campana-T.Peternell. Geometric stability of the cotangent bundle and the universal cover of a projective manifold. Bull. SMF 139 (2011), 41-74.

\bibitem[CP14]{CP14}F. Campana- M. P\u aun. A differential-geometric approach for the generic semi-positivity of orbifold tensor bundles, in preparation.

\bibitem[D97]{D97}J.-P. Demailly.  Algebraic criteria for Kobayashi hyperbolic projective varieties and jet differentials, Proceedings of Symposia in Pure Math., Vol. 62.2 (AMS Summer Institute on Algebraic Geometry, Santa Cruz, July 1995), ed. J. Koll\'ar, R. Lazarsfeld, (1997), 285-360.

\bibitem[D]{D} J.-P. Demailly. Holomorphic morse inequalities and the Green-Griffiths-Lang conjecture; arXiv 1011.3636,
 Pure Appl. Math. Q. 7 (2011), no. 4, Special Issue: \emph{In memory of Eckart Viehweg}.
\smallskip

\bibitem[EV]{EV} H. Esnault, E. Viehweg. Lectures on vanishing theorems. DMV Seminar, 20. BirkhŠuser Verlag, Basel, 1992

\bibitem[Har68]{Ha} Robin Hartshorne. Cohomological dimension of algebraic
varieties. Ann. of Math., 88 (1968), 403Ð450.

\bibitem[Ho]{Ho}A.H\" oring. On a conjecture of Beltrametti and Sommese. arXiv 0912.1295.
\smallskip


\bibitem[J-K]{Jabkeb}K. Jabbusch-S. Kebekus. Positive sheaves of differentials coming from coarse moduli spaces. arxiv 0904.2445.
\smallskip


\bibitem [J-K11]{JK} K. Jabbusch-Stefan Kebekus. Families over special base
manifolds and a conjecture of Campana. Mathematische Zeitschrift 269
(2011), 847Ð878.

\bibitem[K]{K} Y. Kawamata. Subadjunction of log canonical divisors for a subvariety of codimension 2, Contemp. Math., 207, Amer. Math. Soc., Providence, RI, 1997. 
\smallskip

\bibitem[KMM87]{KMM87} Y.Kawamata-K. Matsuki-K.Matsuda. Introduction to the mimimal model
problem. Algebraic Geometry, Sendai 1985. Adv. Stud. Pure Math. 10 (1987),
283 - 360.
\smallskip

\bibitem[Ke]{Kebekus} S. Kebekus. Differential forms on singular spaces, the minimal program, and hyperbolicity of moduli stacks.arxiv 1107.4239
\smallskip

\bibitem[Ke-Ko]{KK} S. Kebekus-S. Kov\`acs. The structure of surfaces and threefolds mapping to the moduli stack of canonically polarised varieties. Duke Math. J. 155 (2010), 1-33.
\smallskip

\bibitem[KK08]{KK'} S. Kebekus-S. Kov\`acs. Families of canonically polarized varieties over surfaces. Inventiones Mathematicae 172 (2008),
657-682.

\bibitem[KST07]{KST}S. Kebekus-L.Sola-Conde-M. Toma. Rationally connected foliations after Bogomolov and McQuillan
(mit Luis Sol‡ und Matei Toma) Journal Algebraic Geometry16 (2007), 65-81.

\bibitem[Kol]{Kol} J. Koll\`ar. Lectures on Resolution of Singularities. (AM-166) (Annals of Mathematics Studies).
\smallskip

\bibitem[La]{La} A. Langer. Logarithmic orbifold Euler numbers of surfaces with applications. arXiv 0012180
\smallskip

\bibitem[Mi85]{Mi} Y. Miyaoka. Deformation of a morphism along a foliation. Algebraic Geometry Bowdoin 1985. Proc. Symp. Pure Math. 46 (1987), 245-268.
\smallskip

\bibitem[MiMo 86]{MiMo86} Y. Miyaoka-S. Mori. A numerical criteria for uniruledness. Ann. Math. 124 (1986), 65-69.
\smallskip

\bibitem[Par]{pars} A. Parshin. Algebraic curves over function fields. Dokl. Nauk. Akad. SSSR 183 (1968), 524-526.
\smallskip

\bibitem[Pat]{Pat} S. Patakfalvi. Viehweg hyperbolicity conjecture is true over compact bases. arXiv 1109.2835
\smallskip


\bibitem[T]{T} B. Taji. The isotriviality of families of canonically-polarised manifolds over a special quasi-projective base. arXiv 1310.5391
\smallskip


\bibitem[V-Z 00]{VZ} E. Viehweg-K. Zuo. Base spaces of non-isotrivial families of smooth minimal models. Complex geometry, G\" ottingen 2000, Springer Verlag. pp. 279-328.

\bibitem[Yau78]{Yau} Yau, S.-T. On the Ricci curvature of a compact Kaehler manifold and the complex Monge-Amp{\`e}re equation. Comm. Pure Appl. Math. 31 (1978)

%\bibitem{}




\end{thebibliography}
\end{document}